\documentclass[10pt,reqno]{amsart}
\usepackage[utf8]{inputenc}
\usepackage{amsopn,amssymb,mathrsfs,xcolor,url,enumitem,accents}
\usepackage[abbrev]{amsrefs} 
\usepackage[a4paper,lmargin=3cm,rmargin=3cm,tmargin=4cm,bmargin=4cm,marginparwidth=2.8cm,marginparsep=1mm]{geometry} 
\usepackage[bookmarks=true,hyperindex,pdftex,colorlinks,citecolor=blue]{hyperref} 
\usepackage{comment, cleveref}

\newtheorem{theorem}{Theorem}[section]
\newtheorem{corollary}[theorem]{Corollary}
\newtheorem{lemma}[theorem]{Lemma}
\newtheorem{proposition}[theorem]{Proposition}

\theoremstyle{definition}

\newtheorem{remark}[theorem]{Remark}

\newcommand{\nm}[1]{\vert\kern-0.25ex\vert #1 \vert\kern-0.25ex\vert}
\newcommand{\C}{\mathbb{C}}
\newcommand{\R}{\mathbb{R}}
\newcommand{\N}{\mathbb{N}}
\newcommand{\diam}{\textnormal{diam}}
\newcommand{\Lip}[1]{\textnormal{Lip}(#1)}
\newcommand{\free}[1]{\mathcal{F}(#1)}
\newcommand{\Lipo}[1]{\textnormal{Lip}_0(#1)}
\newcommand{\lipo}[1]{\textnormal{lip}_0(#1)}
\newcommand{\tildem}{\accentset{\sim}}
\newcommand{\cM}{\mathcal{M}}
\newcommand{\cB}{\mathcal{B}}
\newcommand{\cP}{\mathcal{P}}
\newcommand{\cA}{\mathcal{A}}
\newcommand{\cT}{\mathcal{T}}

\DeclareMathOperator{\Liploc}{Lip}

\numberwithin{equation}{section}

\allowdisplaybreaks

\title{Lipschitz-free spaces over Cantor sets and approximation properties}
\author[F. Talimdjioski]{Filip Talimdjioski}
\address[F. Talimdjioski]{School of Mathematics and Statistics, University College Dublin, Belfield, Dublin 4, Ireland}
\email{filip.talimdjioski@ucdconnect.ie}

\date{\today}

\begin{document}
	
\begin{abstract}
	Let $K=2^\N$ be the Cantor set, let $\cM$ be the set of all metrics $d$ on $K$ that give its usual (product) topology, and equip $\cM$ with the topology of uniform convergence, where the metrics are regarded as functions on $K^2$. We prove that the set of metrics $d\in\cM$ for which the Lipschitz-free space $\free{K,d}$ has the metric approximation property is a residual $F_{\sigma\delta}$ set in $\cM$, and that the set of metrics $d\in\cM$ for which $\free{K,d}$ fails the approximation property is a dense meager set in $\cM$. This answers a question posed by G.~Godefroy.
\end{abstract}
	
\maketitle

\section{Introduction}\label{sect_intro}

Lipschitz-free spaces have become an active research area in Banach space theory in recent years. For a metric space $(M,d)$ and a point $x_0\in M$ we define the Banach space $\Lipo{M,x_0}$ consisting of all real-valued Lipschitz functions $f$ on $M$ that vanish at $x_0$, equipped with the norm $$\nm{f} := \sup_{x,y\in M, x\not= y} \frac{|f(x)-f(y)|}{d(x,y)}.$$ For any $x\in M$ we define the bounded linear functional $\delta_x\in \Lipo{M,x_0}^*$ by $\delta_x(f) = f(x)$, $f\in\Lipo{M,x_0}$. The closed linear span of the set $\{\delta_x : x\in M\}$ is called the \emph{Lipschitz-free space} $\free{M,x_0}$. It is well-known that $\Lipo{M,x_0}$ is isometric to the dual space of $\free{M,x_0}$ and that the Banach space structure of both spaces does not depend on the choice of the base point $x_0$. We will hereafter write $\Lipo{M}$ and $\free{M}$ without specifying the base point, and call $\free{M}$ simply the free space over $M$. In the book \cite{weaver}, Weaver provides a comprehensive introduction to Lipschitz and Lipschitz-free spaces. In it, the latter are called Arens-Eells spaces and are denoted by \AE$(M)$.

One direction of research in this area has been the approximation properties of free spaces. Results involving the approximation property (AP) appear in \cites{godefroyozawa, hajeklancienpernecka, kalton}. Results involving the bounded approximation property (BAP) appear in \cites{godefroykalton, godefroy:15, apandschd}, and metric approximation property (MAP) results appear in \cites{godefroykalton, adalet, godefroyozawa, douchakaufmann, cuthdoucha,fonfw,ps:15,smithtalim}. More details can be found in the introduction to \cite{smithtalim} and in \cite{godefroy:20}.

In particular, we mention that if $M$ is a sufficiently `thin' totally disconnected metric space then $\free{M}$ has the MAP. For example, if $M$ is a countable proper metric space (i.e.~where all closed balls are compact), then $\free{M}$ has the MAP \cite{adalet}. Also, as a corollary of \cite{weaver}*{Corollary 4.39}, $\free{M}$ has the MAP when $M$ is compact and uniformly disconnected (this notion is defined at the beginning of \Cref{sect_prelim}). Moreover, as a corollary of \cite{apandschd}*{Proposition 2.3} and \cite{purely1unrect}*{Theorem B}, $\free{M}$ has the MAP when $M$ is a subset of a finite-dimensional normed space and is purely 1-unrectifiable, which is equivalent to the condition that $M$ contains no bi-Lipschitz image of a compact subset of $\R$ of positive measure. On the other hand, in \cite{godefroyozawa}, G.~Godefroy and N.~Ozawa constructed a compact convex subset $C$ of a separable Banach space such that $\free{C}$ fails the AP. Also, by a result in \cite{hajeklancienpernecka}, there exists a metric space $M$ homeomorphic to the Cantor space such that $\free{M}$ fails the AP.

Given a compact metric space $M$ with metric $d$ and a Lipschitz function $f\colon M\to\R$, we say that $f$ is \emph{locally flat} \cite{weaver}*{Definition 4.1} if for every $x\in M$, $$ \lim_{\epsilon \to 0} \Lip{f|_{B_\epsilon(x)}} = 0,$$ where $B_\epsilon(x) = \{y\in M : d(y,x) < \epsilon\}$. The \emph{Little Lipschitz space} $\lipo{M}$ is the subspace of $\Lipo{M}$ consisting of all locally flat Lipschitz functions. By \cite{weaver}*{Corollary 4.5}, $\lipo{M}$ is a Banach space. It often happens that $\lipo{M} = \{0\}$, for example when $M$ is a connected smooth submanifold of $\R^N$. We say that $\lipo{M}$ \emph{separates points uniformly} \cite{weaver}*{Definition 4.10} if there exists $a\in (0,1]$ such that for every $x,y\in M$ there is $f\in\lipo{M}$ such that $\Lip{f} \leq 1$ and $|f(x)-f(y)| = ad(x,y)$. By \cite{purely1unrect}*{Theorem A, Theorem B}, $\lipo{M}$ separates points uniformly if and only if $\lipo{M}$ is an isometric predual to $\free{M}$, which holds if and only if $M$ is purely 1-unrectifiable. If $M$ is compact and purely 1-unrectifiable, then by \cite{casazzaap}*{Proposition 3.5}, if $\free{M}$ has the MAP then $\lipo{M}$ has the MAP. But also, by the remarks after \cite{godefroysurvey}*{Problem 6.5}, if $\lipo{M}$ has the MAP then $\free{M}$ has the MAP as well.

In \cite{godefroysurvey}, G.~Godefroy surveys various aspects of the theory of free spaces, including the lifting property for separable Banach spaces, approximation properties of free spaces, and norm attainment of Lipschitz functions and operators. Regarding the bounded approximation property of the free space over a compact set $M$, he states a very useful criterion (see the end of \Cref{sect_prelim}) in terms of `almost-extension' operators from Lipschitz spaces over finite subsets of $M$ to $\Lipo{M}$. In the last section he states a number of open problems, several of which concern approximation properties of free spaces.

Let $K=2^\N$ be the Cantor set, equipped with the usual (product) topology. Define $\cM \subseteq C(K^2)$ to be the space of all metrics $d$ on $K$ compatible with its topology. We equip $\cM$ with the metric induced by the usual (supremum) norm of $C(K^2)$. The set $\cM$ is a $G_\delta$ subset of $C(K^2)$ (a proof is provided at the beginning of \Cref{sect_prelim}), and is therefore a Polish (i.e.~separable and completely metrisable) space. For $d\in\cM$ we write $\Lipo{K,d}$ and $\free{K,d}$ for the Lipschitz space and free space over the metric space $(K,d)$, respectively. We define the following subsets of $\cM$:
\begin{align*}
	\cA^\lambda &= \{d\in\cM \; :\; \free{K,d} \text{ has the } \lambda\text{-BAP}\}, \text{ for } \lambda \in [1,\infty), \\
	\cA_f &= \{d\in\cM \; :\; \free{K,d} \text{ fails the AP} \}, \\
	\cP &= \{d\in\cM \; :\; (K,d) \text{ is purely 1-unrectifiable} \}.
\end{align*}
In \cite{godefroysurvey}*{Problem 6.6} Godefroy asks what the topological nature of the set $A_f$ is (it is nonempty by \cite{hajeklancienpernecka}*{Corollary 2.2}). Also, more precisely, he asks whether the set $A_f$ is residual in $\cM$.

In this paper, we investigate the topological properties the subsets of $\cM$ defined above. Section \ref{sect_prelim} concerns notation and preliminary results. In Section \ref{section_top_prop} we prove that the set $\cA^1$ is a residual $F_{\sigma\delta}$ set in $\cM$, and that $\cA_f$ is a dense meager set. Furthermore we prove that $\cP$ is a dense $G_\delta$ set and $\cM\setminus \cP$ is dense. As a corollary we obtain that the set of metrics $d$ for which $\free{M}$ is a dual space to $\lipo{M}$ and both $\free{M}$ and $\lipo{M}$ have the MAP is residual. Finally, in \Cref{sect_lipequiv}, we construct a family $(d_\alpha)$ of metrics on $K$ of size continuum, such that $\Lipo{K,d_\alpha}$ and $\Lipo{K,d_\beta}$ are not isomorphic as algebras for $\alpha\not= \beta$. This should be compared with \cite{hajeklancienpernecka}*{Corollary 2.3}, wherein it is shown that there exists an family $(d_\alpha)$ of metrics on $K$ of size $\aleph_1$, such that $\free{K, d_\alpha}$ is not isomorphic to $\free{K, d_\beta}$ for $\alpha \not= \beta$.

\section{Notation and preliminary results}\label{sect_prelim}

For a metric space $(A,d)$, $x\in A$ and $r>0$ we write $B^d_r(x) = \{y\in A : d(x,y) < r\}$. If $C\subseteq A$ then write $B^d_r(C) = \{y\in A : d(y,C) < r\}$, where $d(y,C) = \inf\{d(y,x) : x\in C\}$. For a real-valued Lipschitz function $f$ on $(A,d)$, $\Liploc_d(f)$ denotes the Lipschitz constant of $f$ with respect to $d$. If $(B,e)$ is another metric space and $f\colon A\to B$ then the Lipschitz constant of $f$ is denoted by $\Liploc_{d,e}(f)$. We call $(A,d)$ and $(B,e)$ \emph{proportional} if there exists a surjection $f\colon A\to B$ and $c>0$ such that $d(x,y) = ce(f(x),f(y))$ for all $x,y\in A$. The space $(A,d)$ is called \emph{uniformly disconnected} \cite{weaver}*{Proposition 4.12} if there exists $r \in (0, 1]$ such that for any distinct $x,y \in A$ there are complementary clopen sets $C,D\subseteq A$ such that $x\in C, y\in D$ and $d(C,D) \geq r d(x,y)$, where $d(C,D) = \inf\{d(z,t) : z\in C, t\in D\}$.

For $d\in\cM$, $K_1, K_2\subseteq K$, and $x\in K$ we put $D(K_1,K_2) = \sup\{d(x,y) : x\in K_1, y\in K_2\}$, and $D(x,K_1) = D(\{x\}, K_1)$. We write $D(K_1)$ or $\diam_d(K_1)$ for $D(K_1,K_1)$. If $S$ is a nonempty set then we call a finite family $\{S_1,\ldots,S_n\}$ of nonempty subsets of $S$ a \emph{partition} of $S$ if $S_i\cap S_j = \emptyset$ whenever $i\not=j$ and $\bigcup_{i=1}^n S_i = S$. If $X$ is a Banach space then $B_X$ denotes the closed unit ball of $X$. If $Y$ is another Banach space then $X\simeq Y$ means that $X$ is isomorphic to $Y$.

We will first provide a short proof of the fact that $\cM$ is a $G_\delta$ set in $C(K^2)$. Let $\mu$ be the canonical metric on $K$: 
\[
\mu(x,y) = \begin{cases}0 & \text{if }x=y,\\ 2^{-n} & \text{if $x\neq y$ and $n\in \N$ is minimal, such that $x(n) \neq y(n)$.} \end{cases}
\] 
Define
\begin{align*}
	A = \{f\in C(K^2) : f(x,x) = f(x,y)-f(y,x) = 0 \leq f(x,y)+f(y,z)-f(x,z) \text{ for all } x,y,z\in K\},
\end{align*}
and for $n\in\N$, define $$B_n = \{f\in C(K^2) \; :\; f(x,y) > 0 \text{ whenever } \mu(x,y) \geq 2^{-n}\}.$$ The set $A$ is clearly closed and hence $G_\delta$ in $C(K^2)$, and the sets $B_n$ are open by compactness. Therefore the set $A \cap \bigcap_{n=1}^\infty B_n$ is $G_\delta$. Pick any $d\in A \cap \bigcap_{n=1}^\infty B_n$ and observe that $d$ is a metric on $K$. Let $\cT_d$ be the topology on $K$ induced by $d$. Since for any $x\in K$, the function $d(x,\cdot)$ is continuous on $K$, $B_r^d(x)$ is open in the topology of $K$ for any $r>0$. Thus $\cT_d$ is a coarser topology than the one on $K$. Therefore any closed subset $C$ of $K$ is compact in $\cT_d$, and is therefore closed, because $\cT_d$ is Hausdorff. Thus $\cT_d$ agrees with the topology on $K$, so $d\in\cM$. As $d$ was arbitrary, $\cM = A \cap \bigcap_{n=1}^\infty B_n$, and so $\cM$ is $G_\delta$.

The following is a crucial lemma that allows us to make small changes to a given metric on $K$ in a very flexible way.

\begin{lemma}\label{lm:mainmetric}
	Suppose that $d\in \cM, \epsilon > 0, \delta \in (0,1]$ and $K'$ is a nonempty clopen subset of $K$. Let $\{K_1',\ldots,K_n'\}$ be an arbitrary partition of $K'$ into clopen sets satisfying $D(K_i') < \frac{\epsilon}{2}$, and let $e_1,\ldots,e_n\in\cM$ be arbitrary. Then there exists $\tildem d\in \cM$ such that $\nm{d-\tildem d}_\infty < \epsilon$, $d$ and $\tildem d$ agree on $K\setminus K'$, $(K_i',\tildem d)$ is proportional to $(K,e_i)$, and $\tildem D(K_i') \leq \delta$ for $i=1,\ldots,n$.
\end{lemma}
\begin{proof}
	Let $d\in\cM, \epsilon > 0,$ and $K'\subseteq K$ be given. Suppose that $\{K_1',\ldots,K_n'\}$ is a partition of $K'$ satisfying 
	\begin{align}\label{eq:diamepsilon2}
		D(K_i') < \frac{\epsilon}{2}, \text{ for } i=1,\ldots,n,
	\end{align}
	and let $e_1,\ldots,e_n\in\cM$ be arbitrary. As each $K_i'$ is homeomorphic to $K$, we can find a metric $d_i$ on $K_i'$ compatible with its topology such that $(K_i',d_i)$ is isometric to $\left(K,\frac{\min(\delta, D(K_i'))}{E_i(K)} e_i\right)$. Hence
	\begin{align}\label{eq:metricdiam2}
		D_i (K_i') = \min(\delta, D(K_i')).
	\end{align}
	Now define $\tildem d\colon K^2\to [0,\infty)$ by
	\begin{equation*}
		\tildem d(x,y) = 
		\begin{cases}
			d(x,y), & \text{if } x,y\in K\setminus K', \\
			D(x,K_i'), & \text{if } x\in K\setminus K' \text{ and } y\in K_i', \\
			D(y,K_i'), & \text{if } x\in K_i'  \text{ and } y\in K\setminus K', \\
			d_i(x,y), & \text{if } x,y\in K_i' \text{ for some } i, \\
			D(K_i',K_j') & \text{if } x\in K_i', y\in K_j' \text{ for } i \text{ and } j \text{ such that } i\not=j.
		\end{cases}
	\end{equation*}
	
	We will now show that $\tildem d$ is a metric on $K$. It is clearly symmetric and satisfies $\tildem d(x,y) = 0$ if and only if $x=y$. To show the triangle inequality, pick $x,y,z\in K$. If $x,y,z\in K\setminus K'$ or $x,y,z\in K_i'$ for some $i$ then the triangle inequality follows from the triangle inequality for the metric $d$ or $d_i$, respectively. We now consider the remaining cases:
	
	Case 1: $x,y\in K\setminus K'$ and $z\in K_i'$ for some $i$. 
	
	We have $$\tildem d(x,y) = d(x,y) \leq d(x,z) + d(z,y) \leq \tildem d(x,z) + \tildem d(z,y).$$ By compactness, there exists $z'\in K_i'$ such that $\tildem d(x,z) = d(x,z')$. Thus $$\tildem d(x,z) = d(x,z') \leq d(x,y) + d(y,z') \leq \tildem d(x,y) + \tildem d(y,z).$$ We can similarly show $\tildem d(y,z) \leq \tildem d(y,x) + \tildem d(x,z)$.
	
	Case 2: $x\in K\setminus K'$ and $y,z\in K_i'$ for some $i$. 
	
	As $\tildem d(x,y) = \tildem d(x,z)$ we have $\tildem d(x,y) \leq \tildem d(x,z) + \tildem d(z,y)$ and $\tildem d(x,z) \leq \tildem d(x,y) + \tildem d(y,z)$. Now let $y',z'\in K_i'$ be such that $D(K_i') = d(y',z')$. We have 
	\begin{align*}
		\tildem d(y,z) = d_i(y,z) &\leq D(K_i') \tag*{by \eqref{eq:metricdiam2}} \\&= d(y',z') \leq d(x,y') + d(z',x) \leq \tildem d(x,y) + \tildem d(z,x).
	\end{align*}
	
	Case 3: $x\in K\setminus K', y\in K_i'$ and $z\in K_j'$ with $i\not=j$. 
	
	Let $y'\in K_i'$ be such that $\tildem d(x,y) = d(x,y')$. We have $$\tildem d(x,y) = d(x,y') \leq d(x,z) + d(z,y') \leq \tildem d(x,z) + \tildem d(z,y).$$ Similarly, $\tildem d(x,z) \leq \tildem d(x,y) + \tildem d(y,z)$. Now let $y_1\in K_i'$ and $z_1\in K_j'$ be such that $\tildem d(y,z) = d(y_1,z_1)$. We have $$\tildem d(y,z) = d(y_1,z_1) \leq d(y_1,x) + d(x,z_1) \leq \tildem d(y,x) + \tildem d(x,z).$$
	
	Case 4: $x,y\in K_i'$ and $z\in K_j'$, where $i\not=j$.
	
	Let $x',y'\in K_i'$ be such that $d(x',y') = D(K_i')$. We have 
	\begin{align*}
		\tildem d(x,y) = d_i(x,y) &\leq D(K_i') \tag*{\text{by} \eqref{eq:metricdiam2}} \\&= d(x',y') \leq d(x',z) + d(z,y') \leq \tildem d(x,z) + \tildem d(z,y).
	\end{align*}
	Also $$\tildem d(x,z) = \tildem d(y,z) \leq \tildem d(x,y) + \tildem d(y,z),$$ and similarly $\tildem d(y,z) \leq \tildem d(y,x) + \tildem d(x,z)$.
	
	Case 5: $x\in K_i',y\in K_j', z\in K_l'$ and $i,j$ and $l$ are distinct. 
	
	Let $x'\in K_i'$ and $y'\in K_j'$ be such that $\tildem d(x,y) = d(x',y')$. We have $$\tildem d(x,y) = d(x',y') \leq d(x',z) + d(z,y') \leq \tildem d(x,z) + \tildem d(z,y).$$
	
	Therefore $\tildem d$ is a metric on $K$. Also, $\tildem d$ is continuous on $K^2$ because it is continuous on each set of the form $K_i'\times K_j'$, $K_i'\times (K\setminus K')$ and $(K\setminus K') \times K_j'$. Therefore, similarly as in the proof that $\cM$ is $G_\delta$, we conclude that $\tildem d\in \cM$. To show $\nm{\tildem d - d}_\infty < \epsilon$, pick $x,y\in K$. If $x,y\in K\setminus K'$ then $\tildem d(x,y) = d(x,y)$. If $x\in K\setminus K'$ and $y\in K_i'$ for some $i$, then let $y'\in K_i'$ be such that $\tildem d(x,y) = d(x,y')$. Then $$|\tildem d(x,y) - d(x,y)| = |d(x,y') - d(x,y)| \leq d(y,y') < \frac{\epsilon}{2},$$ by \eqref{eq:diamepsilon2}. If $x,y\in K_i'$ for some $i$ then $$|\tildem d(x,y) - d(x,y)| \leq \tildem d(x,y) + d(x,y) = d_i(x,y) + d(x,y) < \epsilon,$$ by \eqref{eq:diamepsilon2} and \eqref{eq:metricdiam2}. Finally, if $x\in K_i',y\in K_j'$ for $i\not=j$ then let $x'\in K_i', y'\in K_j'$ be such that $\tildem d(x,y) = d(x',y')$. We have $$|\tildem d(x,y) - d(x,y)| = |d(x',y') - d(x,y)| \leq d(x,x') + d(y,y') < \epsilon,$$ by \eqref{eq:diamepsilon2}. The last assertion of the lemma follows from \eqref{eq:metricdiam2}.
\end{proof}

\begin{remark}\label{rmk:dtildebigger}
	Note that the metric $\tildem d$ satisfies $\tildem d(x,y) \geq d(x,y)$ whenever $x,y\in K$ and $x,y$ do not both belong to $K_i'$ for any $i=1,\ldots,n$.
\end{remark}

\begin{corollary}\label{cor:closemetric}
	Let $d\in\cM, \epsilon > 0,$ and $K'\subseteq K$ be a clopen subset. Then there exists a partition $\{K_1,\ldots,K_n\}$ of $K'$ consisting of clopen sets such that, for any arbitrary $e_1,\ldots,e_n\in\cM$, there exists a metric $\tildem d \in \cM$ such that $\nm{\tildem d - d}_\infty < \epsilon$, $\tildem d$ and $d$ agree on $K\setminus K'$, and $(K_i, \tildem d)$ is proportional to $(K,e_i)$ for each $i=1,\ldots,n$.
\end{corollary}
\begin{proof}
	Let $d\in\cM, \epsilon > 0$ and $K'$ be given. Using the compactness of $K'$ and the fact that each point in $K'$ has a local base for its topology consisting of clopen neighbourhoods, we can find a cover $C_1,\ldots,C_p$ of $K'$ such that $C_i$ is clopen and $D(C_i) < \frac{\epsilon}{2}$ for each $i=1,\ldots,p$. Now inductively define $K_1' = C_1$ and $K_{i+1}' = C_{i+1} \setminus \bigcup_{j=1}^i K_j'$. By dismissing any empty $K_i'$ we obtain a partition $\{K_1',\ldots,K_n'\}$ of $K'$ (where $n\leq p$), consisting of clopen sets satisfying $D(K_i') < \frac{\epsilon}{2}$ for each $i=1,\ldots,n$. Now if $e_1,\ldots,e_n\in\cM$ are arbitrary, the corollary follows from an application of \Cref{lm:mainmetric} with $\delta=1$.
\end{proof}

We will also need the following results.

\begin{lemma}\label{lm:diagonallipschitz}
	Suppose that $d\in \cM$ and $(d_n)_n \subseteq \cM$ are such that $d_n\to d$ uniformly. Let $(M,\rho)$ be a compact metric space, $L>0$ and $h_n\colon K\to M$ be functions such that $\Liploc_{d_n,\rho}(h_n) \leq L$ for each $n\in\N$. Then there exists a subsequence $(h_{n_i})_i$ of $(h_n)_n$ and a function $h\colon K\to M$ such that $\Liploc_{d,\rho} (h) \leq L$ and $h_{n_i} \to h$ uniformly on $K$ as $i\to\infty$. Furthermore,
	\begin{enumerate}
		\item if $J>0$ and $h_n$ is bilipschitz with $\Liploc_{\rho,d_n}(h_n^{-1}) \leq J$ for all $n$, then $h$ is bilipschitz with $\Liploc_{\rho,d} (h^{-1}) \leq J$,
		\item if $h_n$ is surjective for all $n$, then $h$ is surjective.
	\end{enumerate}
\end{lemma}
\begin{proof}
	We will show that the functions $(h_n)_n$ are $d$-$\rho$-equicontinuous. Pick $\epsilon > 0$, and pick $n\in\N$ such that $\nm{d_k-d}_\infty < \frac{\epsilon}{2L}$ for $k\geq n$. Then for $k\geq n$ and $x,y\in K$ such that $d(x,y) < \frac{\epsilon}{2L}$, $$\rho(h_k(x), h_k(y)) \leq Ld_k(x,y) \leq L\left(\frac{\epsilon}{2L} + d(x,y)\right) < \epsilon.$$ Therefore $(h_n)_n$ is equicontiuous on $(K,d)$. By the Arzel\` a-Ascoli theorem, there exists a continuous $h\colon K\to M$ and a subsequence $(h_{n_i})_i$ of $(h_n)_n$ such that $h_{n_i} \to h$ uniformly. For $x,y\in K$, $$\rho(h(x), h(y)) = \lim_{i\to\infty} \rho(h_{n_i}(x), h_{n_i}(y)) \leq \lim_{i\to\infty} Ld_{n_i}(x,y) = Ld(x,y).$$ Hence $\Liploc_{d,\rho}(h) \leq L$.
	
	If $J>0$ and $h_n$ is bilipschitz with $\Liploc_{\rho,d_n}(h_n^{-1}) \leq J$ for all $n\in\N$, then, similarly as previously we can show that $h$ is bilipschitz with $\Liploc_{\rho,d} (h^{-1}) \leq J$. 
	
	Now assume the $h_n$ are surjective. Pick $y\in M$, and let $\epsilon > 0$. Choose $i\in\N$ such that $\nm{h-h_{k_i}}_\infty < \epsilon$. If $h_{k_i}(x)=y$ then $|h(x) - y| < \epsilon$. As $\epsilon$ was arbitrary, $y$ is a limit point of $h(K)$. Since $h(K)$ is compact, $y\in h(K)$ and $h$ is surjective.
\end{proof}

\begin{proposition}\label{freespacesplit}
	Let $d\in \cM,\medspace x_0$ be the base point of $(K,d)$ and $\{K_1,\ldots,K_n\}$ be a partition of $K$ consisting of clopen sets. If $K_i' = K_i\cup \{x_0\}$ for $i=1,\ldots,n$ then $$\free{K,d} \simeq \free{K_1',d} \oplus_1 \ldots \oplus_1 \free{K_n',d}.$$
\end{proposition}
\begin{proof}
	Let $x_0$ be the base point of each $K_i'$ as well. We define the bounded linear operator $T\colon\Lipo{K,\tildem d} \to \Lipo{K_1',d} \oplus_\infty \ldots \oplus_\infty \Lipo{K_n',d}$ by $$T(f) = (f|_{K_1'}, \ldots, f|_{K_n'}).$$ Pick $x,y\in K, x\not=y$. If $x,y\in K_i$ for some $i$ then clearly $\frac{|f(x)-f(y)|}{d(x,y)} \leq \nm{T(f)}$ for each $f\in\Lipo{K,d}$. Otherwise, $d(x,y)>\min_{i\not=j} d(K_i,K_j) =: b$, so $$\frac{|f(x)-f(y)|}{d(x,y)} \leq b^{-1}(|f(x)|+|f(y)|) \leq 2b^{-1} D(K)\nm{T(f)}.$$ Therefore $\nm{f} \leq \max(1,2b^{-1} D(K))\nm{T(f)}$ for each $f\in\Lipo{K,d}$, which shows that $T$ is an isomorphism. It is not hard to see that $T$ is the adjoint of the operator $T_*\colon \free{K_1',d} \oplus_1 \ldots \oplus_1 \free{K_n',d} \to \free{K,d}$, $$T_*(\nu_1,\ldots,\nu_n) = \nu_1 + \ldots + \nu_n,$$ (the spaces $\free{K_i',d})$ are seen as subspaces of $\free{K,d}$, by \cite{weaver}*{Theorem 3.7}). Therefore $T_*$ is the required isomorphism.
\end{proof}

\begin{theorem}[Grothendieck]\label{thm:mapgrothendieck}
	If $X$ is a separable dual Banach space with the AP then $X$ has the MAP.
\end{theorem}

A proof of the previous theorem can be found in e.g.~\cite{lindenstrausstzafriri}*{Theorem 1.e.15}. Finally, in this section we give a useful criterion, due to Godefroy, for the $\lambda$-BAP of free spaces over compact metric spaces in terms of `almost-extension' operators \cite{godefroysurvey}*{Theorem 3.2} (also \cite{godefroy:15}*{Theorem 1}). We will need one of the several equivalent conditions for the $\lambda$-BAP stated in \cite{godefroysurvey}*{Theorem 3.2}. For a finite subset $M'\subseteq M$ of a compact metric space $M$, and $\epsilon > 0$, we say that $M'$ is $\epsilon$-dense in $M$ if for every $y\in M$ there is an $x\in M'$ such that $d(y,x) < \epsilon$. Note that if $(M_n)_n$ is an increasing sequence (with respect to inclusion) of finite subsets of $M$ such that $\bigcup_{n=1}^\infty M_n$ is dense in $M$, then there exists a sequence $(\epsilon_n)_n$ of positive numbers tending to $0$ such that $M_n$ is $\epsilon_n$-dense in $M$ for all $n\in\N$. 

\begin{theorem}\label{thm:godefroycriterion}
	Let $M$ be a compact metric space and $(M_n)_n$ be a sequence of finite $\epsilon_n$-dense subsets of $M$ with $\lim_{n\to\infty} \epsilon_n = 0$. Then $\free{M}$ has the $\lambda$-BAP if and only if there is a sequence $(T_n)_n$ of operators $T_n\colon\Lipo{M_n}\to\Lipo{M}$ such that $\nm{T_n} \leq \lambda$ for all $n$ and $$\lim_{n\to\infty} \sup_{f\in B_{\Lipo{M_n}}} \nm{T_n(f)|_{M_n}-f}_\infty = 0.$$
\end{theorem}

\section{Topological properties of \texorpdfstring{$\cA^\lambda$, $\cA_f$ and $\cP$}{A lambda, A f and P}}\label{section_top_prop}

In this section we give our main results. We will first prove that $\cA^\lambda$ is an $F_{\sigma\delta}$ set in $\cM$ for any $\lambda \geq 1$. Fix a dense sequence of distinct elements $(x_n)_{n=1}^\infty$ in $K$. For $n\in\N$, define $A_n = \{x_1,\ldots,x_n\}$, and for $n\in\N, \lambda\geq 1$ and $\epsilon > 0$, define
\begin{align*}
	\cB _{n,\epsilon}^\lambda = \{d\in \cM \; :\;\; & \text{there exists } T\colon \Lipo{A_n,d}\to\Lipo{K,d} \text{ such that} \\ & \nm{T} \leq \lambda \text{ and } \sup\nolimits_{f\in B_{\Lipo{A_n,d}}} \nm{(Tf)|_{A_n} - f}_\infty \leq \epsilon \}.
\end{align*}

\begin{proposition}\label{setbapclosed}
	The set $\cB _{n,\epsilon}^\lambda$ is closed in $\cM$ for each $n\in\N, \epsilon > 0$ and $\lambda \geq 1$.
\end{proposition}
\begin{proof}
	Fix some $n\in\N, \epsilon > 0$ and $\lambda \geq 1$. Let $(d_k)_k$ be a sequence in $\cB _{n,\epsilon}^\lambda$ converging to $d\in\cM$. For each $k\in\N$, pick some $T_k\colon\Lipo{A_n,d_k}\to\Lipo{K,d_k}$ such that $\nm{T_k} \leq \lambda$ and $$\sup_{f\in B_{\Lipo{A_n,d}}} \nm{(T_kf)|_{A_n} - f}_\infty \leq \epsilon.$$ 
	
	Fix an $m\in\N, m > n$. We consider the sequence $(S_k)_{k=1}^\infty$ where $S_k\colon \Lipo{A_n, d}\to\Lipo{A_m, d}$ is defined by $S_k(f) = T_k(f)|_{A_m}$. As $d_k\to d$ uniformly and $A_m$ is finite, it is not hard to see that $(\nm{S_k})_{k=1}^\infty$ is a bounded sequence. Therefore, by compactness, $(S_k)_k$ has a subsequence $(S_{k_j})_j$ which converges to an operator $P_m\colon \Lipo{A_n, d}\to\Lipo{A_m, d}$. For any $f\in\Lipo{A_n,d}$, we have $\Liploc_{d} (f) = \lim_{j\to\infty} \Liploc_{d_{k_j}}(f)$ and for $x,y\in A_m$,
	\begin{align*}
		|P_m(f)(x) &- P_m(f)(y)| = \lim_{j\to\infty} |S_{k_j}(f)(x) - S_{k_j}(f)(y)| \\&= \lim_{j\to\infty} |T_{k_j}(f)(x) - T_{k_j}(f)(y)| \leq \limsup_{j\to\infty} \Liploc_{d_{k_j}}(T_{k_j}(f)) d_{k_j}(x,y) \\&\leq \limsup_{j\to\infty} \nm{T_{k_j}}\Liploc_{d_{k_j}}(f) d_{k_j}(x,y) \leq \lambda \Liploc_d(f) d(x,y).
	\end{align*}
	Therefore $\nm{P_m}\leq \lambda$. It is also clear that $$\sup_{f\in B_{\Lipo{A_n,d}}} \nm{(P_m(f))|_{A_n} - f}_\infty \leq \epsilon.$$
	
	Now fix a basis $f_1,\ldots,f_{n-1}$ of $\Lipo{A_n,d}$ with $\nm{f_i} = 1$ for all $i\in\{1,\ldots,n-1\}$. For each $i$ and $m>n$ extend the function $P_m(f_i)$ to $(K,d)$ without increasing its Lipschitz constant, by McShane's extension theorem \cite{weaver}*{Theorem 1.33}. Now fix an $i\in\{1,\ldots,n-1\}$ and consider the sequence $(P_m(f_i))_{m=n+1}^\infty \subseteq \lambda B_{\Lipo{K,d}}$. As $\lambda B_{\Lipo{K,d}}$ is $w^*$-compact and $w^*$-metrisable (since $\free{K,d}$ is separable), $(P_m(f_i))_{m=n+1}^\infty$ has a $w^*$-convergent subsequence. As the $w^*$-topology on bounded subsets of $\Lipo{K,d}$ coincides with the topology of pointwise convergence \cite{weaver}*{Theorem 2.37}, we can obtain a subsequence $(P_{m_j})_{j=1}^\infty$ of $(P_m)_{m=n+1}^\infty$ such that $P_{m_j}(f_i)$ converges pointwise to some $T(f_i) \in \lambda B_{\Lipo{K,d}}$ for all $i\in\{1,\ldots,n-1\}$. We extend $T$ by linearity to $\Lipo{A_n,d}$.
	
	For each $p\in\N$ and each $i=\{1,\ldots,n-1\}$, we have $$T(f_i)(x_p) = \lim_{j\to\infty} P_{m_j}(f_i)(x_p).$$ Therefore, by linearity, $T(f)(x_p) = \lim_{j\to\infty} P_{m_j}(f)(x_p)$ for each $f\in\Lipo{A_n,d}$. Now let $p,l \in \N$, $p<l$. Then $x_p,x_l \in A_{m_j}$ whenever $m_j \geq l$, hence
	\[
	|(Tf)(x_p)-(Tf)(x_l)| = \lim_{j \to \infty} |(P_{m_j}f)(x_p)-(P_{m_j}f)(x_l)| \leq \lambda \Liploc_d(f) d(x_p,x_l).
	\]
	From the fact that $(x_p)_{p=1}^\infty$ is dense in $K$ follows that $\Liploc_d(T(f)) \leq \lambda\Liploc_d(f)$, and so $\nm{T}\leq \lambda$. Also, it is not hard to see that $$\sup_{f\in B_{\Lipo{A_n,d}}} \nm{(T(f))|_{A_n} - f}_\infty \leq \epsilon.$$ Therefore $d\in \cB _{n,\epsilon}^\lambda$ and $\cB _{n,\epsilon}^\lambda$ is closed.
\end{proof}

\begin{remark}
	In the second part of the previous proof, extending $P_m(f_i)$ by McShane's extension theorem is done for convenience rather than necessity. Alternatively, we can define $Tf$ only on the set $\{x_n: n\in\N\}$ and then extend to $K$ using the density of $\{x_n: n\in\N\}$.
\end{remark}

\begin{proposition}\label{prop:alambdafsigmadelta}
	The set $\cA^\lambda$ is $F_{\sigma\delta}$ in $\cM$ for any $\lambda\geq 1$.
\end{proposition}
\begin{proof}
	According to \Cref{thm:godefroycriterion} with $M=K$, $M_n = A_n$, $n\in\N$, $\free{K,d}$ has the $\lambda$-BAP if and only if there exists a sequence $(\alpha_n)_n$ of positive real numbers converging to 0 such that $d\in \cB _{n,\alpha_n}^\lambda$ for all $n\in\N$. This is equivalent to $$\forall k \in \N \; \exists n_0\in \N \; \forall n > n_0, \;\; d\in \cB _{n,k^{-1}}^\lambda,$$ or $$d\in \bigcap_{k\in\N} \bigcup_{n_0\in\N} \bigcap_{n>n_0} \cB _{n,k^{-1}}^\lambda.$$ Now the statement follows from \Cref{setbapclosed}.
\end{proof}

We next prove that $\cA^1$ is residual in $\cM$. For $a\in 2^{<\N}$ (that is, $a$ is a finite sequence of 0s and 1s) and $n\in\N$, let $l(a)$ denote the length of $a$, and define $R_n = \{a \in 2^{<\N} \; :\; l(a) = n\}$, and $C_a = \{x\in K \; :\; x|_{\{1,2,\ldots,l(a)\}} = a\}$. Each $C_a$ is a clopen subset of $K$ with $\diam_\mu (C_a) = 2^{-l(a)-1}$, and, for each $n\in\N$, $K$ is the disjoint union of the sets $\{C_a \; :\; a\in R_n\}$. For each $a\in 2^{<\N}$, let $r_a = (a(1),\ldots,a(l(a)),0,0,\ldots)\in K$. For $d\in\cM$ and $n\in\N$, we define $T_n^d\colon\Lipo{\{r_a : a\in R_n\},d}\to\Lipo{K,d}$ by $T_n^d(f)(x) = f(r_{x|_{\{1,2,\ldots,n\}}})$. Also define $$\chi_n^d := \max_{a,b\in R_n, a\not= b} \frac{D(C_a,C_b)}{d(C_a,C_b)}.$$

\begin{lemma}\label{tndnorm}
	The operator $T_n^d$ satisfies $\nm{T_n^d} \leq \chi_n^d$. 
\end{lemma}
\begin{proof}
	Let $f\in \Lipo{\{r_a : a\in R_n\},d}$, $\Lip{f} \leq 1$, and $x,y\in K$. If $x,y\in C_a$ for some $a\in R_n$ then $T_n^d(f)(x) = T_n^d(f)(y)$. If $x\in C_a, y\in C_b$ where $a,b\in R_n, a\not=b$, then $$\frac{|T_n^d(f)(x) - T_n^d(f)(y)|}{d(x,y)} = \frac{|f(r_a) - f(r_b)|}{d(x,y)} \leq \frac{d(r_a,r_b)}{d(x,y)} \leq \frac{D(C_a,C_b)}{d(C_a,C_b)}.$$ Therefore $\nm{T_n^d} \leq \chi_n^d.$
\end{proof}

Define $U_n = \{d\in\cM \; :\; \chi_n^d < 1+\frac{1}{n}\}$ for each $n\in\N$. It is clear that $U_n$ is open in $\cM$.

\begin{proposition}\label{prop:uniondenseopen}
	For each $n_0\in \N$, $\bigcup_{n=n_0}^\infty U_n$ is dense and open in $\cM$.
\end{proposition}
\begin{proof}
	Let $n_0\in\N$. Obviously $\bigcup_{n=n_0}^\infty U_n$ is open. Pick $d\in\cM$ and $\epsilon > 0$. By compactness of $K$, we can find $n\geq n_0$ such that $d(x,y) < \frac{\epsilon}{2}$ whenever $\mu(x,y) < 2^{-n}$, $x,y\in K$. We then have $D(C_a) < \frac{\epsilon}{2}$ for all $a\in R_n$ (strict inequality holds by compactness of $C_a$). Pick $\delta \in \left(0,\frac{1}{2n} \min\limits_{a,b\in R_n, a\not= b} d(C_a,C_b)\right)$, and apply \Cref{lm:mainmetric} to $d, \epsilon, \delta$, $K'=K$, the partition $\{C_a : a\in R_n\}$ of $K'$, and $e_i = \mu$ for all $i$. We obtain a metric $\tildem d$ such that $\nm{\tildem d - d}_\infty < \epsilon$ and $\tildem D(C_a) \leq \delta$ for all $a\in R_n$. By \Cref{rmk:dtildebigger}, 
	\begin{align}\label{eq:dcacb2ndelta}
		\tildem d(C_a, C_b) \geq d(C_a,C_b) > 2n\delta,
	\end{align} 
	for $a,b\in R_n, a\not=b$. Fix some $a,b\in R_n, a\not= b$ and let $x_a,y_a\in C_a$ and $x_b,y_b\in C_b$ be such that $\tildem D(C_a,C_b) = \tildem d(x_a,x_b)$ and $\tildem d(C_a,C_b) = \tildem d(y_a,y_b)$. Then 
	\begin{align*}
		\tildem D(C_a,C_b) &= \tildem d(x_a,x_b) \leq \tildem d(x_a,y_a) + \tildem d(y_a,y_b) + \tildem d(y_b,x_b) \\&\leq \tildem D(C_a) + \tildem D(C_b) + \tildem d(C_a,C_b) \leq 2\delta + \tildem d(C_a,C_b) < \left(1+\frac{1}{n}\right)\tildem d(C_a,C_b),
	\end{align*}
	where the last inequality holds by \eqref{eq:dcacb2ndelta}. As $a,b\in R_n, a\not=b$ were arbitrary, we have $\chi_n^{\tildem d} < 1+\frac{1}{n}$, and so $\tildem d\in U_n$.
\end{proof}

\begin{theorem}\label{a1residual}
	The set $\cA^1$ is a residual $F_{\sigma\delta}$ set in $\cM$.
\end{theorem}
\begin{proof}
	By \Cref{prop:alambdafsigmadelta} it suffices to prove that $\cA^1$ is residual, and by \Cref{prop:uniondenseopen} it suffices to prove $$\bigcap_{n_0=1}^\infty \bigcup_{n=n_0}^\infty U_n \subseteq \cA^1.$$ If $d$ is in the set on the left-hand side, then there exists an increasing sequence of natural numbers $n_1 < n_2 < \ldots$ such that $d\in U_{n_i}$ for all $i$. Then $\nm{T_{n_i}^d} < 1+\frac{1}{n_i}$ for all $i$, by \Cref{tndnorm}. As the set $\{r_a : a\in 2^{<\N}\}$ is dense in $K$, an application of \Cref{thm:godefroycriterion} with $M=K$,  $M_i = \{r_a : a\in R_{n_i}\}$, and operators $\frac{T_{n_i}^d}{\nm{T_{n_i}^d}}$, $i\in\N$, shows that $d\in \cA^1$.
\end{proof}

\begin{proposition}\label{failapdense}
	The set $\cA_f$ is a dense meager set in $\cM$.
\end{proposition}
\begin{proof}
	The fact that $\cA_f$ is meager follows from \Cref{a1residual}. To show it is dense, let $d\in\cM$ and $\epsilon>0$. By \cite{hajeklancienpernecka}*{Corollary 2.2} there exists $d'\in\cA_f$. According to \Cref{cor:closemetric} with $K'=K$ and $e_i = d'$ for all $i$, there exists $\tildem d\in\cM$ and a partition $\{K_1,\ldots,K_n\}$ of $K$ consisting of clopen sets such that $\nm{\tildem d - d}_\infty < \epsilon$ and $(K_i,\tildem d)$ is proportional to $(K,d')$. According to \Cref{freespacesplit}, $$\free{K,\tildem d} \simeq \free{K_1',\tildem d} \oplus_1 \ldots \oplus_1 \free{K_n',\tildem d},$$ where $K_i' = K_i\cup\{x_0\}$ and $x_0$ is the base point of $K$. Since proportional metric spaces have isometrically isomorphic free spaces, $\free{K_1,\tildem d}$ fails the AP. As $\free{K_1,\tildem d}$ has codimension 1 in $\free{K_1', \tildem d}$, we have that $\free{K_1', \tildem d}$ fails the AP as well. Hence $\free{K,\tildem d}$ fails the AP.
\end{proof}

We now state and prove some properties of the set $\cP$ of metrics $d$ for which $(K,d)$ is purely 1-unrectifiable. Denote by $\Delta$ the Lebesgue measure on $\R$.

\begin{proposition}\label{prop:pgdelta}
	The set $\cP$ is $G_\delta$ in $\cM$.
\end{proposition}
\begin{proof}
	For $m,k\in\N$ define
	\begin{align*}
		V_{m,k} = \{& d\in \cM \; :\; \text{there exists a compact } K'\subseteq K \text{ and } h\colon K' \to \R, \text{ such that } \\& m^{-1} d(x,y) \leq |h(x)-h(y)| \leq md(x,y) \text{ for all } x,y\in K', \text{ and } \\& \Delta(h(K')) \geq k^{-1} \}.
	\end{align*}
	We will prove that $V_{m,k}$ is closed. Let $(d_n)_{n\in\N} \subseteq V_{m,k}$ converge uniformly to $d\in\cM$, and let $K_n'$ and $h_n\colon K_n'\to\R$ be the compact set and bilipschitz function, respectively, associated to $d_n$ for each $n\in\N$. By taking a subsequence, we can assume that the $K_n'$ converge to a compact set $K'\subseteq K$ (in the compact-open topology of $\R$). We extend each function $h_n$ to a function (again denoted by $h_n$) on $K$ by \cite{weaver}*{Theorem 1.33}, while preserving its Lipschitz constant with respect to $d_n$. We can assume, by translation, that $h_n(0,0,\ldots) = 0$ for each $n$. As the diameters of $(K,d_n)$ are uniformly bounded, the sets $h_n(K)$ are all contained in some fixed bounded interval. By \Cref{lm:diagonallipschitz}, there exists a subsequence $(h_{n_i})_i$ converging uniformly to a function $h\colon K\to\R$ with $\Liploc_d(h) \leq m$.
	
	Now let $x,y\in K'$ be arbitrary and let $(x_i)_i, (y_i)_i \subseteq K$ converge to $x$ and $y$, respectively, and be such that $x_i,y_i\in K_{n_i}'$ for each $i\in\N$. Let $\epsilon > 0$ and pick $i$ such that $$d(x,x_i), d(y, y_i), \nm{h - h_{n_i}}_\infty, \nm{d-d_{n_i}}_\infty < \epsilon.$$ Then
	\begin{align*}
		|h(x)-h(y)| &\geq |h(x_i) - h(y_i)| - |h(x) - h(x_i)| - |h(y) - h(y_i)| \\&\geq |h(x_i) - h(y_i)| - 2m\epsilon \geq |h_{n_i}(x_i) - h_{n_i}(y_i)| - (2+2m)\epsilon \\&\geq m^{-1} d_{n_i}(x_i,y_i) - (2+2m)\epsilon \geq m^{-1} d(x_i,y_i) - (2+2m+m^{-1})\epsilon \\&\geq m^{-1} (d(x,y)-d(x,x_i)-d(y,y_i)) - (2+2m+m^{-1})\epsilon \\&\geq m^{-1} d(x,y) - (2+2m+3m^{-1})\epsilon.
	\end{align*}
	As $\epsilon$ was arbitrary, we get $|h(x)-h(y)| \geq m^{-1}d(x,y)$. This means that $h$ satisfies the bilipschitz condition in the definition of $V_{m,k}$ on the set $K'$.
	
	To show that $\Delta(h(K')) \geq k^{-1}$, pick $\epsilon > 0$ and set $U = B_\epsilon^{|\cdot|}(h(K'))$. We have that $h^{-1}(U)$ is an open set in $K$ containing $K'$. Choose $i$ such that $K_{n_i}' \subseteq h^{-1}(U)$ and $\nm{h-h_{n_i}}_\infty < \epsilon$. Then $$h_{n_i}(K_{n_i}') \subseteq B_\epsilon^{|\cdot|}(h(K_{n_i}')) \subseteq B_\epsilon^{|\cdot|}(U) \subseteq B_{2\epsilon}^{|\cdot|}(h(K')).$$ Hence $\Delta(B_{2\epsilon}^{|\cdot|}(h(K'))) \geq \Delta(h_{n_i}(K_{n_i}')) \geq k^{-1}$. As $\epsilon$ was arbitrary, we get that $\Delta(h(K')) \geq k^{-1}$. Therefore $d\in V_{m,k}$ and so $V_{m,k}$ is closed. As $\cM\setminus\cP = \bigcup_{m,k\in\N} V_{m,k}$, we have that $\cM\setminus\cP$ is $F_\sigma$, and so $\cP$ is $G_\delta$.
\end{proof}

\begin{proposition}\label{prop:pdense}
	The set $\cP$ is dense in $\cM$.
\end{proposition}
\begin{proof}
	Let $d\in\cM$ and $\epsilon>0$ be arbitrary. It is not hard to see that $(K,\mu)$ is uniformly disconnected. By \cite{weaver}*{Corollary 4.39 (ii)}, $\free{K,\mu}$ is a dual space. Then by \cite{purely1unrect}*{Theorem B}, $(K,\mu)$ is purely 1-unrectifiable. According to \Cref{cor:closemetric} applied to $d$, $\epsilon$, $K' = K$ and $e_i = \mu$ for all $i$, there exists a partition $\{K_1,\ldots,K_n\}$ of $K$ consisting of clopen sets, and a metric $\tildem d\in\cM$, such that $\nm{\tildem d - d}_\infty < \epsilon$ and $(K_i,\tildem d)$ is proportional to $(K,\mu)$ for all $i$. Therefore each $(K_i,\tildem d)$ is purely 1-unrectifiable. Hence, it is not hard to see that $(K,\tildem d)$ is also purely 1-unrectifiable.
\end{proof}

\begin{proposition}\label{mminuspdense}
	The set $\cM \setminus \cP$ is dense in $\cM$.
\end{proposition}
\begin{proof}
	Let $d\in\cM$ and $\epsilon>0$ be arbitrary. Let $d'\in\cM$ be such that $(K,d')$ is isometric to a Cantor subset of $\R$ of positive measure. Then clearly $(K,d')$ is not purely 1-unrectifiable. According to \Cref{cor:closemetric} with $K=K'$ and $e_i = d'$ for all $i$, there exists $\tildem d\in \cM$ such that $\nm{d-\tildem d}_\infty < \epsilon$ and $(K,\tildem d)$ contains a clopen subset proportional to $(K,d')$. Then $(K,\tildem d)$ is not purely 1-unrectifiable.
\end{proof}

\begin{corollary}\label{cor:dualmapresidual}
	The set of metrics $d$ for which $\free{K,d}$ is the dual space of $\lipo{K,d}$ and both $\lipo{M}$ and $\free{M}$ have the MAP is residual in $\cM$.
\end{corollary}
\begin{proof}
	Let $G$ be the set of metrics in question. Clearly $G\subseteq \cA^1$, and by \cite{purely1unrect}*{Theorem B}, $G\subseteq\cP$. Also, if $d\in \cP \cap \cA^1$ then by \cite{casazzaap}*{Proposition 3.5}, $\lipo{K,d}$ has the MAP, so $d\in G$. Therefore $G = \cP \cap \cA^1$ and the corollary follows from \Cref{prop:pgdelta}, \Cref{prop:pdense}, and \Cref{a1residual}.
\end{proof}

\section{Lipschitz equivalence classes}\label{sect_lipequiv}

In this section we consider subsets of $\cM$ consisting of Lipschitz-equivalent metrics, and describe some of their basic topological properties. For any $d\in\cM$, we consider the `Lipschitz equivalence class' of metrics
\begin{align*}
	E_d = \{d'\in\cM \; :\; & \text{ there exists a surjection } h\colon K\to K \text{ and } n > 0 \\&\text{ such that } n^{-1} d(x,y) \leq d'(h(x),h(y)) \leq n d(x,y) \text{ for all } x,y\in K\}.
\end{align*}

\begin{proposition}\label{conteqcl}
	There is a family $(d_\alpha) \subseteq \cM$ of size continuum such that $E_{d_\alpha} \cap E_{d_\beta} = \emptyset$ whenever $\alpha\not=\beta$.
\end{proposition}
\begin{proof}
	For each $\lambda\in(0,1)$, let $d_\lambda\in\cM$ be such that $(K,d_\lambda)$ is isometric to the $\lambda$-middle-thirds Cantor subset of $\R$. By \cite{falconer}*{Exercise 2.14}, the Hausdorff dimension of $(K,d_\lambda)$ for each $\lambda\in(0,1)$ is $\frac{\log 2}{\log\frac{2}{1-\lambda}}$. By \cite{falconer}*{Corollary 2.4} there is no bilipschitz surjection from $(K,d_\lambda)$ to $(K,d_{\lambda'})$ for different $\lambda$ and $\lambda'$, and the statement follows.
\end{proof}

\begin{proposition}
	For each $d\in\cM$, $E_d$ is a dense meager $F_\sigma$ set in $\cM$.
\end{proposition}
\begin{proof}
	Fix $d\in\cM$. If $\hat d\in\cM$ and $\epsilon > 0$ are arbitrary, we will show that there exists a metric $\tildem d\in E_d$ which is $\epsilon$-close to $\hat d$. We apply \Cref{cor:closemetric} to $\hat d$, $K'=K$, and $\epsilon$, to get a partition $K_1,\ldots,K_n$ of $K$ consisting of clopen sets. Now if $e_i \in \cM$ is such that $(K,e_i)$ is isometric to $(K_i,d)$ for $i=1,\ldots,n$, then by the same corollary we obtain a metric $\tildem d\in\cM$, such that $\nm{\hat d -\tildem d}_\infty < \epsilon$, and $(K_i,\tildem d)$ is proportional to $(K,e_i)$, and hence to $(K_i,d)$, for all $i$. It is not hard to see that then $\tildem d \in E_d$. Therefore $E_d$ is dense in $\cM$.
	
	Now for $n\in\N$, consider the set 
	\begin{align*}
		E_d^n = \{d'\in\cM \; :\;\; & \text{there exists a surjection } h\colon K\to K \text{ such that } \\& n^{-1} d(x,y) \leq d'(h(x),h(y)) \leq n d(x,y) \text{ for all } x,y\in K\}.
	\end{align*}
	Suppose that $(d_k)_{k=1}^\infty$ is a sequence in $E_d^n$ converging to $d_0\in\cM$. Let $h_k\colon K\to K$ be the surjection associated with $d_k$, for each $k\in\N$. By \Cref{lm:diagonallipschitz}, there exists a surjective function $h\colon K\to K$ satisfying the condition in the definition of $E_d^n$ with $d_0$ for $d'$. This shows that $d_0\in E_d^n$ and so $E_d^n$ is closed. Since $E_d = \bigcup_{n\in\N} E_d^n$, $E_d$ is $F_\sigma$. Moreover, $E_d^n$ is nowhere dense because if $e\not\in E_d$ then $E_{e}$ is dense in $\cM$ and $E_{e}\cap E_d = \emptyset$. Therefore $E_d$ is meager.
\end{proof}

For a bounded metric space $M$, we define $\Lip{M}$ (resp.~$\Lip{M,\C}$) to be the space of all real (resp.~complex)-valued Lipschitz functions on $M$ equipped with the norm $\nm f = \Lip{f} + \nm{f}_\infty$ (the Lipschitz constant of a complex-valued function is defined by the same supremum as for a real-valued function). Under pointwise multiplication of functions, $\Lip{M}$ and $\Lip{M,\C}$ are Banach algebras. Note that by \cite{weaver}*{Lemma 1.28}, $\Lip{M,\C}$ is isomorphic to $\Lip{M}\oplus \Lip{M}$ as a Banach space. We can also view $\Lipo{M}$ as an algebra under pointwise multiplication, however the Lipschitz constant is not submultiplicative. Instead, it satisfies $\Lip{fg} \leq \nm f _\infty \Lip{g} + \nm g _\infty \Lip{f}$, which implies $\nm{fg} \leq 2\diam(M) \nm f \nm g$. Therefore the product is continuous, which implies that there is an equivalent norm which is submultiplicative.

\begin{corollary}\label{cor:continuumnonisom}
	There exists a family $(d_\alpha) \subseteq\cM$ of size continuum such that $\Lipo{K,d_\alpha}$ and $\Lipo{K,d_\beta}$ are nonisomorphic as algebras for $d_\alpha \not= d_\beta$.
\end{corollary}
\begin{proof}
	Suppose that $d_1,d_2\in\cM$ are such that $\Lipo{K,d_1}$ and $\Lipo{K,d_2}$ are isomorphic as algebras. If $h\colon\Lipo{K,d_1} \to \Lipo{K,d_2}$ is the isomorphism, then the map $\tildem h\colon \Lip{K,d_1}\to\Lip{K,d_2}$, given by $\tildem h(\lambda\textbf{1}_K + f) = \lambda\textbf{1}_K + h(f),$ is an algebra isomorphism, where $\lambda\in\R$, $\textbf{1}_K$ is the constant function on $K$ equal to $1$, and $f\in\Lipo{K,d_1}$. Then the map $\hat{h}\colon \Lip{(K,d_1),\C}\to\Lip{(K,d_2),\C}$, given by $\hat h (f+ig) = \tildem h(f) + i\tildem h(g)$, is again an algebra isomorphism. Finally, by \cite{sherbert}*{Theorem 5.1}, there exists a bilipschitz surjection from $(K,d_1)$ to $(K,d_2)$, and the corollary follows from \Cref{conteqcl}.
\end{proof}

We do not know (in ZFC) whether there exists a family $(d_\alpha) \subseteq\cM$ of size continuum such that $\free{K,d_\alpha}$ is not linearly isomorphic to $\free{K,d_\beta}$ for $\alpha \not= \beta$. 

\subsection*{Acknowledgements}

I would like to thank my supervisor, Dr.~Richard Smith, for the many helpful conversations and suggestions during the research process and the writing of the manuscript.

\bibliography{document}

\end{document}